\documentclass[a4paper,11pt,oneside]{article}
\usepackage[latin1]{inputenc}
\usepackage{amssymb, amsmath}
\usepackage{moreverb}
\usepackage{vmargin}
\usepackage{fancyhdr}
\usepackage{amsfonts}
\usepackage{bm}
\usepackage{fancybox}
\usepackage{amsthm}
\usepackage{amsmath}
\newtheorem{theorem}{Theorem}[section]

\newtheorem{lemma}[theorem]{Lemma}

\newtheorem{definition}[theorem]{Definition}

\usepackage[
  linktocpage,
  urlbordercolor={1 0.5 0.125},
  filebordercolor={1 0.5 0.125},
  linkbordercolor={0.25 1 0.25},
  citebordercolor={0.25 1 0.25},
  pdfborderstyle={/S/U/W 1}, 
  breaklinks]{hyperref}  
  
\begin{document}
\ \begin{center}
\begin{LARGE}
\bf {Two-dimensional individual\\ clustering
 model} \\
\end{LARGE}
\vspace{0.5cm}
Elissar Nasreddine\\
\vspace{0.2cm}
 \begin{small}
\textit{Institut de Math\'ematiques de Toulouse, Universit\'e de Toulouse,}\\
 \textit{F--31062 Toulouse cedex 9, France}\\
 \vspace{0.1cm}
 e-mail: elissar.nasreddine@math.univ-toulouse.fr\\
 \today
\end{small}
\end{center}
\textbf{Abstract}: This paper is devoted to study a model of individual clustering with two specific reproduction rates in two space dimensions. Given $q>2$ and an initial condition in $W^{1,q}(\Omega)$, the local existence and uniqueness of solution have been shown in \cite{well}. In this paper we give a detailed proof of existence of global solution.
\section{Introduction}
In the present work, we deal with a model of individual dispersing of individual with an additional aggregation mechanism introduced in \cite{models}. Given a sufficiently smooth function $E$, parameters $\delta\in (0,1)$, $\varepsilon \geq0$ and $r\geq 0$, the equations take the form
\begin{equation}
\label{grin5}
\left\{
\begin{array}{llll}
\displaystyle \partial_t u&=& \delta\ \Delta u-\nabla\cdot(u\ \bm{\omega})+r\ u\ E(u),& x\in \Omega, t>0 \\
\displaystyle -\varepsilon\ \Delta \bm{\omega}+\bm{\omega}&=& \nabla E(u),& x\in \Omega, t>0 \\
\end{array}
\right.
\end{equation}
in an open bounded domain $\Omega\subset \mathbb{R}^2$, where  $u(t,x)>0$, $\bm{\omega}(t,x)\in \mathbb{R}^2$ and $E$ denote the population density, the average velocity of dispersing individuals, and the individual net reproduction rate, respectively. In this model, the individuals are assumed to disperse randomly in space ($\delta\ \Delta u$) with a bias $-\nabla\cdot (u\ \bm{\omega})$ in the direction of increasing reproduction rate, the term  $\varepsilon \ \Delta \bm{\omega}$ acting as a mollifier to smooth out any sharp local variation in $\nabla E(u)$.\\

In \cite{models}, we supplement \eqref{grin5} with no-flux boundary conditions
\begin{equation}
\label{in1.5}
\partial_{\bm{n}} u=\bm{n}\cdot \bm{\omega}=0, \ \ x\in\partial \Omega, \ t\geq 0,
\end{equation}
where $\bm{n}$ is the outward unit normal of $\partial \Omega$ and $\partial_{\bm{n}} u=\bm{n}\cdot \nabla u$. However, to guarantee the well-posedness of the elliptic system for $\bm{\omega}$ in two space dimensions, we should append the following condition given in \cite{asimplified, unprobleme, quelques}
\begin{equation}
\label{in2.5}
\partial_{\bm{n}} \bm{\omega}\times \bm{n}=0, \ \ x\in \partial \Omega, \ t\geq0,
\end{equation}
where $\partial_{\bm{n}} \bm{\omega}= (\partial_{\bm{n}} \omega_1 , \partial_{\bm{n}} \omega_2) = (\bm{n}\cdot \nabla\omega_1, \bm{n}\cdot\nabla\omega_2)$ for the vector field $\bm{\omega}=(\omega_1,\omega_2)$. in other words, \eqref{in2.5} means that $\partial_{\bm{n}} \bm{\omega} $ is parallel to $\bm{n}$, where $\bm{v}\times \bm{u}=v_1\ u_2 - u_1\ v_2 $.\\

Given $q>2$, and an initial condition $u_0\in W^{1,q}(\Omega)$, the existence and uniqueness of a nonnegative and maximal solution of \eqref{grin5}, \eqref{in1.5} and \eqref{in2.5} have been shown in \cite{well}, and the purpose of this paper is to prove the global existence of solution when $E(u)$ has the two specific forms suggested in \cite{models},  namely
\begin{equation}
\label{in3.5}
E(u)=(1-u)\ (u-a)
\end{equation} 
for some $a\in (0,1)$, or 
\begin{equation}
\label{in4.5}
E(u)=1-u.
\end{equation}
For these choices of reproduction rates, global existence has been shown in \cite{well} in one space dimension and the purpose of this work is to prove that the solutions are global as well in two space dimensions. As in the one-dimensional case, the starting point of the analysis is an $L^\infty(L^2)$ estimate on $u$ and an $L^2$ estimate on $\nabla \cdot \bm{\omega}$. Combining the latter with Gagliardo-Nirenberg inequality gives $L^\infty(L^p)$ estimates on $u$ for any $p>2$. This then allow us to obtain an $L^\infty(L^2)$ bound on $\nabla u$ which in turn gives an $L^\infty$ bound on $\bm{\omega}$ by elliptic regularity.\\

The paper is organized as follows. In section 2, we state the global existence results, and focus on the two
specific forms of $E$: The ``bistable case" \eqref{in3.5} see Theorem \ref{th1.5}, and the ``monostable case" \eqref{in4.5}, see Theorem \ref{th2.5}. In section 3, we recall the local existence result obtained in \cite{well} and we give some properties of the elliptic system for $\bm{\omega}$. In section 4, we turn to the global existence issue in the bistable case. The proof starts from the $L^\infty(L^2)$ estimate for $u$, an $L^2$ estimate for $\bm{\omega}$ and an $L^2$ estimate on $\nabla \cdot \bm{\omega}$ obtained from a suitable cancellation between the coupling terms in the $u$ and $\bm{\omega}$ equations, then, for $p>2$, we derive an $L^\infty(L^p)$ estimate for $u$. Then we use Lemma A.1 of \cite{Boundedness} to derive an $L^\infty$ estimate of $u$ and we end the proof by an $L^\infty(L^q)$ estimate on $\nabla u$. This ensures global existence. In section 5, we prove the global existence in the monostable case. The proof is quite similar to that of the previous case, except for the first estimate.
\section{Main result}
We first define the notion of solution to \eqref{grin5}-\eqref{in2.5} to be used in this paper.
\begin{definition}\label{def1.5}
Let $T>0$, $q> 2$,  and an initial condition $u_0\in W^{1, q}(\Omega)$ . A strong solution of \eqref{grin5}-\eqref{in2.5} on $[0, T)$ is a function
$$u \in C \left( [0,T), W^{1,q}(\Omega)\right)\cap C\left( (0,T), W^{2,q}(\Omega)\right),$$ such that
\begin{equation}
\label{eqdeu5}
\left\{
\begin{array}{llll}
\displaystyle\partial_t u&=& \delta\ \Delta u-\nabla\cdot(u\ \bm{\omega})+r\ u\ E(u),& \mathrm{a.e.\ in }\ [0,T)\times\Omega\\
\displaystyle u(0,x)&=&u_0(x),& \mathrm{a.e.\ in }\ \Omega\\
\displaystyle \partial_{\bm{n}}u&=&0,& \mathrm{a.e.\ on }\ [0,T)\times\partial\Omega,
\end{array}
\right.
\end{equation}
where, for all $t\in [0,T)$, $\bm{\omega}(t)$ is the unique solution in $W^{2,q}(\Omega)$ of
\begin{equation}
\label{eqdefi5}
\left\{
\begin{array}{llll}
\displaystyle-\varepsilon\Delta \bm{\omega}(t)+\bm{\omega}(t)&=& \nabla E(u(t))& \mathrm{a.e.\ in}\ \Omega\\
\displaystyle\bm{\omega}(t)\cdot \bm{n}=\partial_{\bm{n}}\bm{\omega}(t) \times \bm{n}&=&0& \mathrm{a.e.\ on}\ \partial \Omega
\end{array}
\right.
\end{equation}
\end{definition}

In the following theorem we give the global existence of solution to \eqref{grin5}-\eqref{in2.5} in the bistable case, that is when $E(u)=(1-u)(u-a)$, for some $a\in (0,1)$.
\begin{theorem} \label{th1.5}Let $q>2$, and assume that $u_0$ is a nonnegative function in $W^{1,q}(\Omega)$,\\ and $E(u)=(1-u)(u-a)$ for some $a\in (0,1)$. Then \eqref{grin5}-\eqref{in2.5} has a global nonnegative solution $u$ in the sense of Definition \ref{def1.5}.
\end{theorem}
The proof starts with a suitable cancellation of the coupling terms in the two equations which gives an estimate for $u$ in $L^\infty (L^2)$ and for $\nabla \cdot \bm{\omega}$ in $L^2$. Using the Gagliardo-Nirenberg inequality \eqref{gag5} we derive, for $p>2$ an $L^\infty(L^p)$ estimate for $u$. Then by the regularity properties of the second equation in \eqref{grin5} we obtain an $L^\infty$ bound of $\bm{\omega}$. Combining these estimates and Lemma A.1 of \cite{Boundedness} provide us with an $L^\infty$ estimate for $u$ which is used to show an $L^\infty(L^q)$ estimate for $\nabla u$. This proves that the solution cannot explode in finite time.\\

Next, we turn to the global existence issue in the monostable case, that is  when $E(u)=1-u$.
\begin{theorem} \label{th2.5}Let $q>2$, and assume that $u_0$ is a nonnegative function in $W^{1,q}(\Omega)$,\\ and $E(u)=(1-u)$. Then \eqref{grin5}-\eqref{in2.5} has a global nonnegative solution $u$ in the sense of Definition \ref{def1.5}.
\end{theorem}
The proof of the previous theorem follows the same lines as that of Theorem \ref{th1.5}. As in the bistable case, there is a cancellation between the two equations which provide us an $L^\infty(L\log L)$ bound on $u$ and an $L^2$ bound for $\nabla \cdot\bm{\omega}$ as a starting point.

\section{Well-posedness}
Throughout this paper and unless otherwise stated, we assume that 
$$ \delta\in (0,1),\ \varepsilon>0,\ r\geq0.$$
We first recall some properties of the strong solution of the following system,
\begin{equation}
\label{ellip5}
\left\{
\begin{array}{llll}
\displaystyle -\varepsilon\ \Delta \bm{\omega}+\bm{ \omega}&=&f,& \mathrm{in}\ \Omega, \\
\displaystyle \bm{\omega}\cdot \bm{n}&=&0,& \mathrm{on}\ \partial \Omega,\\
\displaystyle \partial_{\bm{n}} \bm{\omega}\times \bm{n}&=&0,& \mathrm{on}\ \partial \Omega,
\end{array}
\right.
\end{equation}
where $f\in (L^p(\Omega))^2$ and $p>1$. The strong solutions of \eqref{ellip5} is solving \eqref{ellip5} a.e. in $\Omega$. In this direction the existence and uniqueness of the strong solution to \eqref{ellip5} are proved in \cite{quelques}:
\begin{theorem}\label{reg5}
For $f\in (L^p(\Omega))^2$ with $1<p<\infty$, \eqref{ellip5} has a unique solution in $(W^{2, p}(\Omega))^2$ such that
\begin{equation}\label{regularite5}||\bm{\omega}||_{W^{2, p}}\leq \frac{K(p)}{\varepsilon} \ ||f||_{p},\end{equation}
where $K(p)=K(p,\Omega)$.
\end{theorem}
In other words, the strong solution has the same regularity as elliptic equations with classical boundary conditions.\\

Thanks to \cite{well} we recall the existence and uniqueness result of the maximal solution of \eqref{grin5}-\eqref{in2.5}.
\begin{theorem}\label{local5}
We assume that $E\in C^2(\mathbb{R})$, let $p>2$ and a nonnegative function\\ $u_0\in W^{1,p}(\Omega)$. Then, for some $T_{\mathrm{max}}\in (0,\infty]$, there is a unique nonnegative maximal solution \begin{equation}\label{l1.5}u\in C\left( [0, T_{\mathrm{max}}), W^{1,p}(\Omega)\right)\cap C\left( (0, T_{\mathrm{max}}), W^{2,p}(\Omega)\right) \end{equation}
 to \eqref{grin5}-\eqref{in2.5} in the sense of Definition \ref{def1.5}. 
Moreover, if for each $T>0$, there is $C(T)$ such that
$$||u(t)||_{W^{1, p}}\leq C(T),\ \ \mathrm{for\ all}\ t\in[0,T]\cap[0,T_{\mathrm{max}}),$$
then $T_{\mathrm{max}}=\infty$. In addition, $u$ satisfies
\begin{equation}\label{eq1.5}
u(t,x)=\left( e^{t(\delta\ \Delta)}\ u_0(x)\right)+\int_0^t e^{(t-s)(\delta\ \Delta)}\ \left[-\nabla\cdot(u\ \bm{\omega})+r\ u\ E(u)\right](s,x)\ ds,
\end{equation}
for $(t,x)\in [0,T_{\mathrm{max}}]\times \Omega$, where $\left(e^{t\ (\delta\ \Delta)}\right)$ denotes the semigroup generated in $L^p(\Omega)$ by $\delta \  \Delta$ with homogeneous Neumann boundary conditions.
\end{theorem}
We recall that there is  $C>0$ such that
\begin{equation}\label{eq2.5}
||e^{t\ (\delta\ \Delta)} v||_{W^{1,p}}\leq C \ ||v||_{W^{1,p}},\ \mathrm{and}\ ||\nabla e^{t\ (\delta\ \Delta)} v ||_p\leq C\ \delta^{-\frac{1}{2}}\ t^{-\frac{1}{2}}\ ||v||_p.
\end{equation}
Also in several places we shall need the following Gagliardo-Nirenberg inequality
\begin{equation}
\label{gag5}
||u||_p\leq C\ ||u||_{W^{1,2}}^\theta\ ||u||_q^{1-\theta},\ \ \mathrm{with}\ \theta=\frac{p-q}{p}, \ u\in W^{1,2}(\Omega)
\end{equation}
which holds for all $p\geq 1$ and $q\in [1,p]$.
Also we use the following singular Gronwall lemma (see \cite[Theorem 3.3.1]{linear}).
\begin{lemma}\label{Gronwall}
Given $ \alpha, \beta\in [0,1)$ , there exists a positive constant $c:= c(\alpha, \beta)$ such that the following is true:\\

If $f:\ (0,T)\longrightarrow \mathbb{R}$ satisfies
\begin{equation}\label{gr1.5}
\left[t\mapsto t^\beta\ f(t)\right]\in L^\infty_{\mathrm{loc}}((0,T), \mathbb{R}),
\end{equation}
and
\begin{equation}
f(t)\leq A\ t^{-\beta}+B\ \int_0^t \frac{1}{(t-s)^{\alpha}}\ f(s)\ ds,\ \ a.a.t \in (0,T),
\end{equation}
 where $A$ and $B$ are positive constants, then
$f(t)\leq C(T)$, for all $t\in (0,T)$, where $C$ depends only on $T,\ \alpha,\ \beta,$ and $\gamma$.
\end{lemma}

\section{Global existence}
\subsection{The bistable case: $E(u)=(1-u)(u-a)$}

We recall the system 
\begin{equation}
\label{bc5}
\left\{
\begin{array}{llll}
\displaystyle \partial_t u&=& \delta\ \Delta u-\nabla\cdot(u\ \bm{\omega})+r\ u\ (1-u)\ (u-a),& x\in \Omega, t>0 \\
\displaystyle -\varepsilon\ \Delta \bm{\omega}+\bm{\omega}&=& [-2\ u+(a+1)]\ \nabla u,& x\in \Omega, t>0 \\
\displaystyle \partial_{\bm{n}}u=0&,& \bm{\omega}\cdot \bm{n}=\partial_{\bm{n}} \bm{\omega}\times \bm{n}=0,& x\in \partial \Omega, t>0\\
\displaystyle u(0,x)&=&u_0(x),& x\in \Omega.
\end{array}
\right.
\end{equation}
for a some $a\in (0,1)$, and $u_0\in W^{1,q}(\Omega)$ for some $q>2$.\\

Since $E\in C^2(\mathbb{R})$, Theorem \ref{local5} ensures that there is a maximal solution of \eqref{bc5} in $C\left( [0, T_{\mathrm{max}}), W^{1,q}(\Omega)\right)\cap C\left( (0, T_{\mathrm{max}}), W^{2,q}(\Omega)\right) $  for $q>2$.\\

We begin the proof by the following lemmas which gives some estimates on $u$ and $\bm{\omega}$.

\begin{lemma}\label{le1.5}
Let the same assumptions as that of Theorem \ref{th1.5} hold, and $u$ be the nonnegative maximal solution of \eqref{bc5}. Then for all $T>0$ there exists $C_1(T)>0$, such that $u$ and $\bm{\omega}$ satisfy the following estimates
\begin{equation}\label{dxu5}
||u(t)||^2_2+\int_0^t ||\nabla u(s)||_2^2\ ds\leq C_1(T),\ \ \mathrm{for\ all}\ t\in [0,T]\cap[0, T_{\mathrm{max}}),
\end{equation}
and
\begin{equation}\label{fi5}
\int_0^t \left(||\nabla \cdot \bm{\omega}(s)||^2_{2}+||\bm{\omega}(s)||_2^2\right)\ ds\leq C_1(T)\ \ \mathrm{for\ all}\ t\in [0,T]\cap[0, T_{\mathrm{max}}).
\end{equation}
\end{lemma}
\begin{proof}
We multiply the first equation in \eqref{bc5} by $2\ u$ and integrate it over $\Omega$, to obtain
\begin{equation}\label{cub5}
\frac{\mathrm{d}}{\mathrm{d}t} \int_{\Omega}|u|^2\ dx=-2\ \delta \int_{\Omega}|\nabla u|^2\ dx+2\  \int_{\Omega}\ u\ \bm{\omega}\cdot\nabla u\ dx+2\ r\  \int_{\Omega} u^2\ E(u)\ dx.
\end{equation}
We multiply now the second equation in \eqref{bc5} by $\bm{\omega}$ and integrate it over $\Omega$. We note that the boundary conditions for $\bm{\omega}$ guarantee that $\bm{\omega}$ is tangent to $\partial\Omega$ while $\partial_{\bm{n}}\bm{\omega}$ is normal to $\partial\Omega$. Consequently, $\partial_{\bm{n}} \bm{\omega}\cdot \bm{\omega}=0$ on $\partial\Omega$ and it follows from an integration by parts that
\begin{eqnarray*}
-\varepsilon \int_{\Omega} \Delta \bm{\omega}\ \cdot \bm{\omega}\ dx&=&\varepsilon\  \int_{\Omega}|\nabla\cdot \bm{\omega}|^2\ dx-\varepsilon \int_{\partial\Omega}\left[\ (\nabla \omega_1\cdot \bm{n})\ \omega_1+(\nabla \omega_2\cdot \bm{n})\ \omega_2\ \right] \ d\sigma\\
&=&\varepsilon\  \int_{\Omega}|\nabla\cdot \bm{\omega}|^2\ dx.
\end{eqnarray*}
We thus obtain
\begin{equation}\label{phi5}
\varepsilon\  \int_{\Omega}|\nabla\cdot \bm{\omega}|^2\ dx+ \int_{\Omega}|\bm{\omega}|^2\ dx=-2\  \int_{\Omega} u\ \bm{\omega}\cdot\nabla u\  dx+(a+1)\int_{\Omega}\bm{\omega}\cdot\nabla u\ dx.
\end{equation}
At this point we notice that the cubic terms on the right hand side of \eqref{cub5} and \eqref{phi5} cancel one with the other, and summing \eqref{phi5} and \eqref{cub5} we obtain 
\begin{equation*}
\frac{\mathrm{d}}{\mathrm{d}t}||u||^2_2+\varepsilon\ ||\nabla\cdot \bm{\omega}||_2^2+||\bm{\omega}||^2_2+2\ \delta\ ||\nabla u||_2^2=2\ r \int_{\Omega} u^2\ E(u)\ dx+ (a+1)\int_{\Omega} \bm{\omega}\cdot\nabla u\ dx.
\end{equation*}
We integrate by parts and use Cauchy-Schwarz inequality to obtain
\begin{equation}
(a+1)\ \int_{\Omega} \bm{\omega}\cdot\nabla u\ dx=-(a+1)\int_{\Omega} u\  \nabla\cdot \bm{\omega}\ dx\nonumber\\
\leq\frac{(a+1)^2}{2\ \varepsilon}\ ||u||_2^2+\frac{\varepsilon}{2}\ ||\nabla \cdot \bm{\omega}||_2^2.\nonumber
\end{equation}
On the other hand, $u^2\ E(u)\leq 0$ if $u\notin (a,1)$ so that
$$\int_{\Omega} u^2\ E(u)\ dx\leq |\Omega|\ (1-a).$$
The previous inequalities give 
\begin{equation*}\frac{\mathrm{d}}{\mathrm{d}t}||u||^2_2+\frac{\varepsilon}{2}\ ||\nabla\cdot \bm{\omega}||_2^2+||\bm{\omega}||^2_2+2\ \delta\ ||\nabla u||_2^2\leq\frac{(a+1)^2}{2\ \varepsilon}\ ||u||_2^2+ 2\ |\Omega|\ r\ (1-a).\end{equation*}
Therefore, for all $T>0$ there exists $C_1(T)$ such that \eqref{dxu5} and \eqref{fi5} hold.
\end{proof}
\begin{lemma}\label{le2.5}
Let the same assumptions as that of Theorem \ref{th1.5} hold, and $u$ be the nonnegative maximal solution of \eqref{bc5}. Then for all $T>0$ there exists $C_2(T,p)>0$, such that for $p\geq2$
\begin{equation}\label{eq3.5}
||u(t)||_p\leq C_2(T,p)\ \ \mathrm{for\ all}\ t\in [0,T]\cap[0, T_{\mathrm{max}}),
\end{equation}
\begin{equation}\label{eq4.5}
\int_{0}^t||\nabla u^{\frac{p}{2}}(s)||_2^2\ ds\leq C_2(T,p)\ \ \mathrm{for\ all}\ t\in [0,T]\cap[0, T_{\mathrm{max}}).
\end{equation}
\end{lemma}
\begin{proof}
We multiply the first equation in \eqref{bc5} by $p\ u^{p-1}$, integrate with respect to $x$, and integrate by parts. The boundary terms vanish and we obtain
\begin{eqnarray}
\frac{\mathrm{d}}{\mathrm{d}t} ||u||_p^p&\leq& \frac{-4\ \delta\ (p-1)}{p}\ ||\nabla u^{\frac{p}{2}}||_2^2-(p-1)\int_{\Omega} \nabla\cdot{\bm{\omega}}\ u^p\ dx\nonumber\\
&+& r\ p\int_{\Omega} u^{p-1}\ E(u)\ dx.\nonumber
\end{eqnarray}
By Cauchy-Schwarz inequality we obtain
\begin{eqnarray}
\frac{\mathrm{d}}{\mathrm{d}t} ||u||_p^p&\leq& \frac{-4\ \delta\ (p-1)}{p}\ ||\nabla u^{\frac{p}{2}}||_2^2+(p-1)\ ||\nabla\cdot\bm{\omega}||_2\ ||u^\frac{p}{2}||_4^2\label{eq6.5}\\
&+& r\ p\ (1-a)\ |\Omega|.\nonumber
\end{eqnarray}
Using the Gagliardo-Nirenberg inequality \eqref{gag5} we have
\begin{equation}\label{eq5.5}
||u^{\frac{p}{2}}||_4\leq C\ ||u^{\frac{p}{2}}||_{W^{1,2}}^{\frac{1}{2}}\ ||u^{\frac{p}{2}}||_2^{\frac{1}{2}}.
\end{equation}
Substituting \eqref{eq5.5} in \eqref{eq6.5}, and by Young inequality we obtain
\begin{eqnarray}
\frac{\mathrm{d}}{\mathrm{d}t} ||u||_p^p&\leq& \frac{-4\ \delta\ (p-1)}{p}\ ||\nabla u^{\frac{p}{2}}||_2^2+C\ (p-1)\ ||\nabla\cdot\bm{\omega}||_2\ ||u^{\frac{p}{2}}||_{W^{1,2}}\ ||u^{\frac{p}{2}}||_2\nonumber\\
&+& r\ p\ (1-a)\ |\Omega|\nonumber\\
&\leq& \frac{-4\ \delta\ (p-1)}{p}\ ||\nabla u^{\frac{p}{2}}||_2^2+\frac{2\ \delta\ (p-1)}{p}\ ||u^{\frac{p}{2}}||_{W^{1,2}}^2\nonumber\\
&+&C(p)\ ||\nabla\cdot\bm{\omega}||_2^2\ \ ||u^{\frac{p}{2}}||^2_2+C(p)\nonumber\\
&\leq& \frac{-2\ \delta\ (p-1)}{p}\ ||\nabla u^{\frac{p}{2}}||_2^2+ C(p)\ ||u||_p^p+C(p)\ ||\nabla\cdot\bm{\omega}||_2^2\ \ ||u||^p_p+C(p).\nonumber
\end{eqnarray}
Next, integrating the above inequality in time, and using \eqref{fi5} yield that there exists $C_2(T,p)$ such that \eqref{eq3.5} and \eqref{eq4.5} hold.
\end{proof}

\begin{lemma}\label{le3.5}
Let the same assumptions as that of Theorem \ref{th1.5} hold, and $u$ be the nonnegative maximal solution of \eqref{bc5}. Then for all $T>0$ there exists $C_3(T)>0$, such that 
\begin{equation}\label{eq7.5}
||\nabla u(t)||_2\leq C_3(T)\ \ \mathrm{for\ all}\ t\in [0,T]\cap[0, T_{\mathrm{max}}).
\end{equation}
\end{lemma}
\begin{proof}
We multiply the first equation in \eqref{bc5} by $-\Delta u$, integrate over $\Omega$, and use Cauchy-Schwarz, and Young inequalities and \eqref{eq3.5} to obtain
\begin{eqnarray}
\frac{1}{2}\frac{\mathrm{d}}{\mathrm{d}t}||\nabla u||_2^2&=&-\delta \ ||\Delta u||_2^2+\int_\Omega \left( \nabla u\cdot \bm{\omega}+\nabla\cdot \bm{\omega}\  u\right)\ \Delta u\ dx
- r\int_\Omega u\ (1-u)\ (u-a)\ \Delta u\ dx\nonumber\\
&\leq& -\delta\ ||\Delta u||_2^2+\frac{\delta}{2}\ ||\Delta u||_2^2+C\ \int_{\Omega} |\nabla u|^2\ |\bm{\omega}|^2\ dx\nonumber\\
&+&C\ \int_{\Omega} |\nabla\cdot \bm{\omega}|^2\ |u|^2\ dx+C\ r\ ||u\ (1-u)\ (u-a)||_2^2+\frac{\delta}{4}||\Delta u||_2^2\nonumber\\
&\leq& -\frac{\delta}{4}\ ||\Delta u||_2^2+C\ \int_{\Omega} |\nabla u|^2\ |\bm{\omega}|^2\ dx+C\ \int_{\Omega} |\nabla\cdot \bm{\omega}|^2\ |u|^2\ dx+C(T).\label{eq8.5}
\end{eqnarray}
To go further requires to improve the estimate on $\bm{\omega}$ and $\nabla \cdot \bm{\omega}$. For that purpose, we use Lemma \ref{le1.5} and Lemma \ref{le2.5} for $p=4$ to obtain for all $T>0$

\begin{equation}\label{eq9.5}
\int_{0}^t ||\nabla E(u)(s)||_2^2\ ds\leq (a+1)^2\int_{0}^t ||\nabla u(s)||_2^2\ ds+ 4\int_0^t ||u\ \nabla u(s)||_2^2\ ds\leq C_2(T),
\end{equation}
for all $t\in [0,T]\cap[0, T_{\mathrm{max}})$.
Consequently, $\nabla E(u)$ is bounded in $L^2\left( (0,t)\times \Omega\right)$. By Theorem \ref{reg5} and the continuous embedding of $W^{2,2}(\Omega)$  in $W^{1,4}(\Omega)$, and $W^{1,4}(\Omega)$ in $L^\infty(\Omega)$, we have
\begin{equation*}
||\bm{\omega}||_\infty+||\nabla \cdot\bm{\omega}||_4\leq C\ || \bm{\omega}||_{W^{1,4}} \leq C\ ||\bm{\omega}(s)||_{W^{2,2}}\leq C\ || \nabla E(u)||_2,
\end{equation*}
 which together with \eqref{eq9.5} implies that
\begin{equation}\label{eq10.5}\int_0^t \left(||\bm{\omega}(s)||^2_\infty+||\nabla \cdot\bm{\omega}(s)||^2_4\right)\ ds\leq C\ \int_0^t ||\nabla E(u)(s)||_2^2\ ds\leq C_2(T),\end{equation}
for all  $t\in [0,T]\cap[0, T_{\mathrm{max}})$. Using H\"older Inequality , \eqref{eq8.5} becomes
\begin{equation*}
\frac{1}{2}\frac{\mathrm{d}}{\mathrm{d}t}||\nabla u||_2^2\leq-\frac{\delta}{4}\ ||\Delta u||_2^2+C\ ||\bm{\omega}||_\infty^2 \int_{\Omega} |\nabla u|^2\ dx+C\  ||\nabla\cdot \bm{\omega}||_4^2\ ||u||_4^2+C(T).
\end{equation*}
Next we integrate the above inequality in time, and use \eqref{eq3.5} for $p=4$, and \eqref{eq10.5}  to obtain
$$||\nabla u(t)||_2^2\leq ||\nabla u_0||_2^2+C\ \int_0^t ||\bm{\omega}(s)||_\infty^2\ ||\nabla u(s)||_2^2\ ds+C(T),$$
using \eqref{eq10.5} again, we have thus proved \eqref{eq7.5}.
\end{proof}
\begin{lemma}\label{le4.5}
Let the same assumptions as that of Theorem \ref{th1.5} hold, and $u$ be the nonnegative maximal solution of \eqref{bc5}. Then for all $T>0$ there exists $C_4(T)>0$, such that 
\begin{equation}\label{eq11.5}
||\bm{\omega}(t)||_\infty\leq C_4(T), \ \mathrm{for\ all}\ t\in [0,T]\cap[0, T_{\mathrm{max}}).
\end{equation}
\end{lemma}
\begin{proof}
It follows from \eqref{eq3.5}, \eqref{eq7.5} and H\"older inequality,  that there exists $C(T)>0$ such that 
 \begin{equation*}||(-2u+a+1)\ \nabla u||_{\frac{3}{2}}\leq C(T)\ ||-2\ u+a+1||_6\ ||\nabla u||_2\leq C(T).\end{equation*} Consequently, Theorem \ref{reg5} ensures that 
\begin{equation*}
||\bm{\omega}||_{W^{2,\frac{3}{2}}}\leq C\ \left\lvert\left\lvert[-2\ u+a+1]\ \nabla u\right\rvert\right\rvert_{\frac{3}{2}} \leq C(T), \ \mathrm{for\ all}\ t\in [0,T]\cap[0, T_{\mathrm{max}}).
\end{equation*}
Using the continuous embedding of $W^{2,\frac{3}{2}}(\Omega)$ in $L^\infty(\Omega)$ we have thus proved \eqref{eq11.5}.
\end{proof}
Next, thanks to lemma A.1 in \cite{Boundedness} we can derive a uniform bound for $u$.
\begin{lemma}\label{le5.5}
Let the same assumptions as that of Theorem \ref{th1.5} hold, and $u$ be the nonnegative maximal solution of \eqref{bc5}. Then for all $T>0$ there exists $C_5(T)>0$, such that 
\begin{equation}\label{eq12.5}
||u(t)||_\infty\leq C_5(T), \ \mathrm{for\ all}\ t\in [0,T]\cap[0, T_{\mathrm{max}}).
\end{equation}
\end{lemma}
\begin{proof} We can see that the function $u$ solves
\begin{equation}
\label{eqdeu}
\left\{
\begin{array}{llll}
\displaystyle\partial_t u&=& \nabla\cdot(\delta\ \nabla u)+\nabla\cdot f+g,& \mathrm{a.e.\ in }\ [0,T_{\mathrm{max}})\times\Omega\\
\displaystyle u(0,x)&=&u_0(x),& \mathrm{a.e.\ in }\ \Omega\\
\displaystyle \partial_nu&=&0,& \mathrm{a.e.\ on }\ [0,T_{\mathrm{max}})\times\partial\Omega,
\end{array}
\right.
\end{equation}
where $f=-u\ \bm{\omega}$ and $g= r\ u\ E(u)$.\\

By the regularity \eqref{l1.5} of $u$, and by the continuous embeddings of $W^{1,p}(\Omega)$ in $C(\bar{\Omega})$ and of $W^{2,p}(\Omega)$ in $C(\bar{\Omega})$ for $p>2$ we obtain that $f=-u\ \bm{\omega} \in C\left( (0,T_{\mathrm{max}}); C(\bar{\Omega})\right)$ and $\nabla f \in C\left( (0,T_{\mathrm{max}}); C(\bar{\Omega})\right)$. Using \eqref{l1.5} and the continuous embeddings of $W^{1,p}(\Omega)$ in $C(\bar{\Omega})$ for $p>2$ again, we see that $g= r\ u\ E(u)\in C\left((0,T_{\mathrm{max}})\times\bar{\Omega}\right)$. On the other hand, $f\cdot n=-u\ \bm{\omega}\cdot n=0$ on $\partial \Omega\times (0,T_{\mathrm{max}})$.\\
 
 Thanks to \eqref{eq3.5}, we have 
 $$||u(t)||_{p_0}\leq C(T),\ \mathrm{for\ all}\ t\in [0,T]\cap[0, T_{\mathrm{max}}) $$
 where $p_0=9$, while \eqref{eq3.5} and \eqref{eq11.5} yield 
 \begin{equation*}
 ||f(t)||_{q_1}\leq C(T)\  \mathrm{and}\ \ ||g||_{q_2}\leq C(T),\ \mathrm{for\ all}\ t\in [0,T]\cap[0, T_{\mathrm{max}})
 \end{equation*}
where $q_1>4$ and $q_2=3>2$. Since $p_0>1-\frac{3q_1-4}{q_1-3}$ and $p_0=9>0$, we can apply lemma A.1 in \cite{Boundedness}, and the estimate \eqref{eq12.5} holds.
 \end{proof}
\begin{lemma}\label{le6.5}
Let the same assumptions as that of Theorem \ref{th1.5} hold, and $u$ be the nonnegative maximal solution of \eqref{bc5}. Then for all $T>0$ there exists $C_6(T)>0$, such that
\begin{equation}\label{eq20.5}
||\nabla u(t)||_q\leq C_6(T), \ \mathrm{for\ all}\ t\in [0,T]\cap[0, T_{\mathrm{max}}).
\end{equation} 
\end{lemma}
\begin{proof}
Using \eqref{eq1.5},\eqref{eq2.5}, \eqref{eq12.5} and \eqref{eq11.5} we have for $q>2$
\begin{eqnarray}
||\nabla u(t)||_q&\leq& C\ ||u_0||_{W^{1,q}}+C_1\ \int_0^t (t-s)^{-\frac{1}{2}}\  ||\nabla\cdot(u\ \bm{\omega})(s)||_q\ ds\nonumber\\
&+&r\ C_1\ \int_0^t\ ||\nabla (u\ E(u))(s)||_q\ ds\nonumber\\
&\leq&C\ ||u_0||_{W^{1,q}}+C_1\ \int_0^t (t-s)^{-\frac{1}{2}}\ ||u(s)||_\infty\ ||\nabla\cdot \bm{\omega}(s)||_q\ ds\nonumber\\
&+&C_1\ \int_0^t (t-s)^{-\frac{1}{2}} ||\bm{\omega}(s)||_\infty\ ||\nabla u(s)||_q\ ds\label{eq14.5}\\
&+&r\ C_1\ \int_0^t ||(-3\ u^2+2\ (a+1)\ u-a)(s)||_\infty\ ||\nabla u(s)||_q\ ds.\nonumber\\
&\leq& C(T)\ \left( 1+\int_0^t(t-s)^{-\frac{1}{2}}(||\nabla\cdot \bm{\omega}(s)||_q+||\nabla u(s)||_q)\ ds+\int_0^t ||\nabla u(s)||_q\ ds\right).\nonumber
\end{eqnarray}
By Theorem \ref{reg5} we have
\begin{equation}\label{eq13.5}
||\nabla\cdot \bm{\omega}(s)||_q\leq ||\bm{\omega}(s)||_{W^{2,q}}\leq K(q)\ ||(-2u+a+1)(s)||_\infty\ ||\nabla u(s)||_q.
\end{equation}
Substituting \eqref{eq13.5} into \eqref{eq14.5} and using \eqref{eq12.5} we obtain
\begin{eqnarray*}
||\nabla u(t)||_q\leq C(T)\ \left( 1+\int_0^t(t-s)^{-\frac{1}{2}}\ ||\nabla u(s)||_q\ ds+\int_0^t ||\nabla u(s)||_q\ ds\right).
\end{eqnarray*}
Now, we obtain \eqref{eq20.5} by applying Lemma  \ref{Gronwall} with $\beta=0$,  and $\alpha=\frac{1}{2}$.
\end{proof} 
Thanks to these lemmas we can now prove the main Theorem \ref{th1.5}
\begin{proof}[Proof of Theorem \ref{th1.5}]
For all $T>0$, Lemma \ref{le2.5} and Lemma \ref{le5.5} ensure that for $q>2$
\begin{equation*}
||u(t)||_{W^{1,q}}\leq C(T),  \ \mathrm{for\ all}\ t\in [0,T]\cap[0, T_{\mathrm{max}})
\end{equation*}
which guarantees that $u$ cannot explode in $W^{1,q}(\Omega)$ in finite time and thus\\ that $T_{\mathrm{max}}=\infty$.
\end{proof}
\subsection{Monostable case}
For this choice of $E$, system \eqref{grin5}-\eqref{in2.5} now reads
\begin{equation}
\label{mc5}
\left\{
\begin{array}{llll}
\displaystyle \partial_t u&=& \delta\ \Delta u-\nabla\cdot(u\ \bm{\omega})+r\ u\ (1-u),& x\in \Omega, t>0 \\
\displaystyle -\varepsilon\ \Delta \bm{\omega}+\bm{\omega}&=& -\ \nabla u,& x\in \Omega, t>0 \\
\displaystyle \partial_{\bm{n}}u=0&,& \bm{\omega}\cdot\bm{n}=\partial_{\bm{n}} \bm{\omega}\times \bm{n}=0,& x\in \partial \Omega, t>0\\
\displaystyle u(0,x)&=&u_0(x),& x\in \Omega,
\end{array}
\right.
\end{equation}
when $u_0\in W^{1,q}(\Omega)$ for some $q>2$. 
Since $E\in C^2(\mathbb{R})$, Theorem \ref{local5} ensures that there is a maximal solution of \eqref{mc5} in $C\left( [0, T_{\mathrm{max}}), W^{1,q}(\Omega)\right)\cap C\left( (0, T_{\mathrm{max}}), W^{2,q}(\Omega)\right) $  for $q>2$.\\

To prove Theorem \ref{th2.5} we need to prove the following lemma
\begin{lemma}\label{le7.5}
Let the same assumptions as that of Theorem \ref{th2.5} hold, and let  $u$ be the maximal solution of \eqref{mc5}. Then for all $T>0$, there exists $C_7(T)>0$ such that

\begin{equation}\label{fiw5}
\int_0^t \left(||\bm{\omega}(s)||^2_{2}+||\nabla \cdot \bm{\omega}(s)||_2^2\right)\ ds\leq C_7(T),\ \ \mathrm{for\ all}\ t\in [0,T]\cap[0, T_{\mathrm{max}}).
\end{equation}
\end{lemma}
\begin{proof}
The proof goes as follows. On the one hand, we multiply the first equation in \eqref{mc5} by $(\log u+1)$ and integrate it over $\Omega$, the boundary terms vanish. Since $u\ (1-u)\ \log u\leq 0$ and $u\ (1-u)\leq 1$,
\begin{eqnarray}
\frac{\mathrm{d}}{\mathrm{d}t}\int_{\Omega} u\ \log u\ dx&=& -\int_\Omega (\delta\ \nabla u-u\ \bm{\omega})\cdot (\frac{1}{u}\ \nabla u)\ dx+r\ \int_\Omega u\ (1-u)\ (\log u+1)\ dx\nonumber\\
&\leq& -\int_\Omega \frac{\delta}{u}\ |\nabla u|^2\ dx+\int_\Omega \bm{\omega}\cdot \nabla u\ dx+|\Omega|\ r.\label{uu5}
\end{eqnarray}
On the other hand, we multiply the second equation in \eqref{mc5} by $\bm{\omega}$, integrate it over $\Omega$. We note that the boundary conditions for $\bm{\omega}$ guarantee that $\bm{\omega}$ is tangent to $\partial\Omega$ while $\partial_{\bm{n}}\bm{\omega}$ is normal to $\partial\Omega$. Consequently, $\partial_{\bm{n}} \bm{\omega}\cdot \bm{\omega}=0$ on $\partial\Omega$ and it follows from an integration by parts that
\begin{equation}\label{fi25}\varepsilon \int_\Omega |\nabla\cdot \bm{\omega}|^2\ dx+\int_\Omega |\bm{\omega}|^2\ dx=-\int_\Omega \bm{\omega}\cdot\nabla u\ \ dx.\end{equation}
Adding \eqref{fi25} and \eqref{uu5} yields
\begin{equation}\label{hb5}
\frac{\mathrm{d}}{\mathrm{d}t}\int_\Omega u\ \log u\ dx+\varepsilon\ ||\nabla\cdot \bm{\omega}||_2^2+||\bm{\omega}||_2^2\leq -4\ \delta\int_\Omega |\nabla \sqrt{u}|^2\ dx+|\Omega|\ r.
\end{equation}
Finally, \eqref{fiw5} is obtained by a time integration of \eqref{hb5}.
\end{proof}
\begin{proof}[Proof of Theorem \ref{th2.5}]
Thanks to \eqref{fiw5}, we now argue as in the proof of Lemma \ref{le2.5}, Lemma \ref{le3.5}, Lemma \ref{le4.5}, Lemma \ref{le5.5} and Lemma \ref{le6.5} to get that for $q>2$
\begin{equation*}
||u(t)||_\infty+||\nabla u(t)||_q\leq C(T),\ \ \mathrm{for\ all}\ t\in [0,T]\cap[0, T_{\mathrm{max}}).
\end{equation*}
Thus, the maximal solution $u$ of \eqref{mc5} cannot explode in finite time.
\end{proof}
\section*{Acknowledgment}
I thank Philippe Lauren\c cot for fruitful discussion.

\end{document}